\documentclass[10pt,reqno]{amsart}

 \setbox0=\hbox{$+$}
\newdimen\plusheight
\plusheight=\ht0
\def\+{\;\lower\plusheight\hbox{$+$}\;}

\setbox0=\hbox{$-$}
\newdimen\minusheight
\minusheight=\ht0
\def\-{\;\lower\minusheight\hbox{$-$}\;}

\setbox0=\hbox{$\cdots$}
\newdimen\cdotsheight
\cdotsheight=\plusheight
\def\cds{\lower\cdotsheight\hbox{$\cdots$}}

\renewcommand{\(}{\left\(}
\renewcommand{\)}{\right\)}

\renewcommand{\pmod}[1]{\,(\textup{mod}\,#1)}
\numberwithin{equation}{section}
\theoremstyle{plain}

\theoremstyle{plain}
\newtheorem{thm}{Theorem}[section]

\theoremstyle{definition}

\theoremstyle{remark}

\numberwithin{equation}{section}

\begin{document}
		\begin{center} {\bf \large Some Identities of Ramanujan's q-Continued Fractions of  Order Fourteen and Twenty-Eight, and Vanishing Coefficients}\vskip 5mm
	
			{\bf Shraddha Rajkhowa and Nipen Saikia}\\\vspace{.1 cm}			
			Department of Mathematics, Rajiv Gandhi University,\\ Rono Hills, Doimukh-791112, Arunachal Pradesh, India.\\
			E. Mails: shraddha.rajkhowa@rgu.ac.in; nipennak@yahoo.com
					\end{center}\vskip2mm

		\noindent {\bf Abstract:}	We deduce $q$-continued fractions $S_{1}(q)$, $S_{2}(q)$ and $S_{3}(q)$ of order fourteen, and continued fractions $V_{1}(q)$, $V_{2}(q)$ and $V_{3}(q)$ of order twenty-eight from a general continued fraction identity of Ramanujan. We establish some theta-function identities  for the continued fractions and derive some colour partition identities as applications. Some  vanishing coefficients results arising from the continued fractions are also offered.  
		\vskip 2mm
		
		\noindent {\bf Keywords and phrases:} $q$-continued fraction; theta-functions; colour partition of integer;  vanishing coefficient.\vspace{.3cm}
		
		\noindent{\bf 2000 Mathematical Subject Classification: }11A55; 11F27; 11P84.

\section{\bf Introduction}
For any complex numbers $\lambda$ and $q$, define the $q$-product $(\lambda;q)_\infty$ as
\begin{equation}\label{aq}\left(\lambda;q\right)_\infty:=\prod_{t=0}^{\infty}\left(1-\lambda q^{t}\right), \qquad |q|<1.\end{equation} For brevity, we often write
$$(\lambda_1;q)_\infty(\lambda_2;q)_\infty(\lambda_3;q)_\infty\cdots(\lambda_m;q)_\infty=\left(\lambda_1, \lambda_2, \lambda_3, \cdots, \lambda_m;q\right)_\infty.$$
The Ramanujan's general theta-function $\mathfrak{f}(\frak a, \frak b)$ \cite[p. 34]{bcb3} is defined as
\begin{equation}\label{eq2}
\mathfrak{f}(\frak a,\frak b)=\sum_{t=-\infty}^{\infty}\frak a^{{t(t+1)/2}}\frak b^{{t(t-1)/2}},\qquad |\frak{ab}|<1.
\end{equation}
In terms of $\mathfrak{f}(\frak a, \frak b)$,  Jacobi's triple product identity  \cite[p. 35, Entry 19]{bcb3} can be stated as
\begin{equation}
\label{eqjn1}\mathfrak{f}(\frak a,\frak b)=(-\frak a;\frak {ab})_\infty(-\frak b;\frak {ab})_\infty(\frak {ab};\frak {ab})_\infty=(-\frak a,-\frak b, \frak {ab}; \frak {ab})_\infty.
\end{equation}Three useful special cases of $\mathfrak{f}(\frak a,\frak b)$ are the theta-functions $\phi(q)$,  $\psi(q)$ and  $f(-q)$ \cite[p. 36, Entry 22(i)-(iii)]{bcb3} given by
\begin{align}\label{phidef}\phi(q)&:=\mathfrak{f}(q,q)=\sum_{t=-\infty}^\infty q^{t^2}=\frac{(-q;-q)_\infty}{(q;-q)_\infty},\\
	\label{psidef}\psi(q)&:=\mathfrak{f}(q,q^3)=\sum_{t=0}^\infty q^{t(t+1)/2}=\frac{(q^2;q^2)_\infty}{(q;q^2)_\infty},\\
	\label{fqdef}f(-q)&:=\mathfrak{f}(-q,-q^2)=\sum_{t=-\infty}^\infty(-1)^tq^{t(3t-1)/2}=(q;q)_\infty.
\end{align}Ramanujan also defined the function $\chi(q)$ \cite[p. 36, Entry 22(iv)]{bcb3} as
\begin{equation}
\label{chidef} \chi(q)=(-q;q^2)_\infty.
\end{equation}

One of the Ramanujan's remarkable contribution is in the field of $q$-continued fractions. Ramanujan recorded many continued fractions in his notebooks and most famous among is the Rogers-Ramanujan continued fraction $R(q)$ defined by 
\begin{equation}\label{rq}
R(q):=q^{1/5}\frac{(q, q^4;q^5)_\infty}{(q^2, q^3;q^5)_\infty}=q^{1/5}\frac{\mathfrak{f}\left(-q,-q^4\right)}{\mathfrak{f}\left(-q^2,-q^3\right)}=\frac{q^{1/5}}{1+\dfrac{q}{1+\dfrac{q^2}{1+\dfrac{q^3}{1+\cdots}}}}, \qquad |q|<1. 
\end{equation} The Rogers-Ramanujan continued fraction $R(q)$ is often referred as the continued fraction of order five. Ramanujan also offered some theta-function identities and modular relations  for the continued fraction $R(q)$. An account of these can be found in \cite{bcb5}. Ramanujan also recorded some general continued fraction identities in his notebook. For example, Ramanujan recorded the following general continued fraction identity \cite[p. 24, Entry 12]{bcb3}: 
Suppose that $a$, $b$ and $q$ are complex numbers with  $|ab|<1$ and $|q|<1$, or that $a=b^{2m+1}$ for some integer $m$. Then
\label{2}\begin{equation}\label{gc}
\frac{(a^2q^3;q^4)_\infty(b^2q^3;q^4)_\infty}{(a^2q;q^4)_\infty(b^2q;q^4)_\infty}=\dfrac{1}{1-ab+\dfrac{(a-bq)(b-aq)}{(1-ab)(q^2+1)+\dfrac{(a-bq^3)(b-aq^3)}{(1-ab)(q^4+1)+\cdots}}}.
\end{equation} By specialising the values of $a$ and $b$,  and taking suitable powers of $q$, one can obtain $q$-continued fractions of particular order which satisfy theta-function identities analogous to those of $R(q)$.

In this paper, we deal with the $q$-continued fractions of order fourteen and twenty-eight.  By replacing  $q$ by $q^{7/2}$ in \eqref{gc}, setting $\{a=q^{1/4}, b=q^{13/4}\}$, $\{a=q^{3/4}, b=q^{11/4}\}$ and $\{a=q^{5/4}, b=q^{9/4}\}$ and simplifying using the results $\{(q^{17}; q^{14})_\infty=(q^3;q^{14})_\infty/(1-q^3)\}$,  $\{(q^{16};q^{14})_\infty=(q^2;q^{14})_\infty/(1-q^2)\}$ and $\{(q^{15};q^{14})_\infty=(q;q^{14})_\infty/(1-q)\}$, we obtain  the following three continued fractions of order fourteen, respectively.:
$$\hspace{-7cm}S_{1}(q):=q^{1/4}\frac{(q^3, q^{11}; q^{14})_\infty}{(q^{4}, q^{10}; q^{14})_\infty}=q^{1/4}\frac{\mathfrak{f}(-q^3, -q^{11})}{\mathfrak{f}(-q^{4}, -q^{10})}$$\begin{equation}\label{S1q}\hspace{-2.2cm}=\dfrac{q^{1/4}(1-q^3)}{(1-q^{7/2})+\dfrac{q^{7/2}(1-q^{1/2})(1-q^{13/2})}{(1-q ^{7/2})(1+q^{7})+\dfrac{q^{7/2}(1-q^{15/2})(1-q^{27/2})}{(1-q^{7/2})(1+q^{14})+\cdots}}},
\end{equation}
$$\hspace{-7cm}S_{2}(q):=q^{3/4}\frac{(q^2, q^{12}; q^{14})_\infty}{(q^{5}, q^{9}; q^{14})_\infty}=q^{3/4}\frac{\mathfrak{f}(-q^2, -q^{12})}{\mathfrak{f}(-q^{5}, -q^{9})}$$\begin{equation}\label{S2q}\hspace{-2.2cm}=\dfrac{q^{3/4}(1-q^2)}{(1-q^{7/2})+\dfrac{q^{7/2}(1-q^{3/2})(1-q^{11/2})}{(1-q ^{7/2})(1+q^{7})+\dfrac{q^{7/2}(1-q^{17/2})(1-q^{25/2})}{(1-q^{7/2})(1+q^{14})+\cdots}}}
\end{equation} and
$$\hspace{-7cm}S_{3}(q):=q^{5/4}\frac{(q, q^{13}; q^{14})_\infty}{(q^{6}, q^{8}; q^{14})_\infty}=q^{5/4}\frac{\mathfrak{f}(-q, -q^{13})}{\mathfrak{f}(-q^{6}, -q^{8})}$$\begin{equation}\label{S3q}\hspace{-2.2cm}=\dfrac{q^{5/4}(1-q)}{(1-q^{7/2})+\dfrac{q^{7/2}(1-q^{9/2})(1-q^{5/2})}{(1-q ^{7/2})(1+q^{7})+\dfrac{q^{7/2}(1-q^{19/2})(1-q^{23/2})}{(1-q^{7/2})(1+q^{14})+\cdots}}}. 
\end{equation} 

Similarly,  we derive continued fractions  $V_{1}(q)$, $V_{2}(q)$ and $V_{3}(q)$ of order twenty-eight from \eqref{gc} which are given by  
$$\hspace{-7cm}V_{1}(q):=q^{3}\frac{(q, q^{27}; q^{28})_\infty}{(q^{13}, q^{15}; q^{28})_\infty}=q^{3}\frac{\mathfrak{f}(-q, -q^{27})}{\mathfrak{f}(-q^{13}, -q^{15})}$$\begin{equation}\label{V1q}\hspace{-2.2cm}=\dfrac{q^{3}(1-q)}{(1-q^{7})+\dfrac{q^{7}(1-q^{6})(1-q^{8})}{(1-q ^{7})(1+q^{14})+\dfrac{q^{7}(1-q^{20})(1-q^{22})}{(1-q^{7})(1+q^{28})+\cdots}}},
\end{equation}
$$\hspace{-7cm}V_{2}(q):=q^{2}\frac{(q^3, q^{25}; q^{28})_\infty}{(q^{11}, q^{17}; q^{28})_\infty}=q^{2}\frac{\mathfrak{f}(-q^3, -q^{25})}{\mathfrak{f}(-q^{11}, -q^{17})}$$\begin{equation}\label{V2q}\hspace{-2.2cm}=\dfrac{q^{2}(1-q^3)}{(1-q^{7})+\dfrac{q^{7}(1-q^{4})(1-q^{10})}{(1-q ^{7})(1+q^{14})+\dfrac{q^{7}(1-q^{18})(1-q^{24})}{(1-q^{7})(1+q^{28})+\cdots}}}
\end{equation} and
$$\hspace{-7cm}V_{3}(q):=q\frac{(q^5, q^{23}; q^{28})_\infty}{(q^{9}, q^{19}; q^{28})_\infty}=q\frac{\mathfrak{f}(-q^5, -q^{23})}{\mathfrak{f}(-q^{9}, -q^{19})}$$\begin{equation}\label{V3q}\hspace{-2.2cm}=\dfrac{q(1-q^5)}{(1-q^{7})+\dfrac{q^{7}(1-q^{2})(1-q^{12})}{(1-q ^{7})(1+q^{14})+\dfrac{q^{7}(1-q^{16})(1-q^{26})}{(1-q^{7})(1+q^{28})+\cdots}}}.
\end{equation}
To obtain $V_{1}(q)$, $V_{2}(q)$ and $V_{3}(q)$, we replace $q$ by $q^{7}$ in \eqref{gc}, then set $\{a=q^{3}, b=q^{4}\}$, $\{a=q^{2}, b=q^{5}\}$ and $\{a=q, b=q^{6}\}$  and then simplify using the results that $\{(q^{29}; q^{28})_\infty=(q;q^{28})_\infty/(1-q)\}$,  $\{(q^{31};q^{28})_\infty=(q^3;q^{28})_\infty/(1-q^3)\}$ and $\{(q^{33};q^{28})_\infty=(q^5;q^{28})_\infty/(1-q^5)\}$, respectively.

In Sect. 2, we prove some theta-function identities for the continued fractions $S_{1}(q)$, $S_{2}(q)$, $S_{3}(q)$, $V_{1}(q)$, $V_{2}(q)$ and $V_{3}(q)$. Using colour partition of integers we deduce some  partition-theoretic results from the theta-function identities in Sect. 3. Some results on vanishing coefficients arising from the continued fractions are also discussed.

\section{\bf Theta-function identities for $S_{i}(q)$ and $V_{i}(q)$} 
This section is devoted to prove some theta-function identities for the continued fractions $S_{1}(q)$, $S_{2}(q)$, $S_{3}(q)$, $V_{1}(q)$, $V_{2}(q)$ and $V_{3}(q)$. 
\begin{thm}\label{y1}We have 
	\begin{align} \hspace{-2.1cm}(i)&\quad\frac{1}{S_{1}(q)}-S_{1}(q)=\frac{\phi(q^{7/2})\frak f(-q^{1/2},-q^{13/2})}{q^{1/4}\psi(q^{7})\frak f(-q^3,-q^{4})},\notag\\
		\hspace{-2.1cm}(ii)&\quad\frac{1}{S_{1}(q)}+S_{1}(q)=\frac{\phi(-q^{7/2})\frak f(q^{1/2},q^{13/2})}{q^{1/4}\psi(q^{7})\frak f(-q^3,-q^{4})},\notag\\
		\hspace{-2.1cm}(iii)&\quad\frac{1}{S_{2}(q)}-S_{2}(q)=\frac{\phi(q^{7/2})\frak
			f(-q^{3/2},-q^{11/2})}{q^{3/4}\psi(q^{7})\frak f(-q^2,-q^5)},\notag\\
		\hspace{-2.1cm}(iv)&\quad\frac{1}{S_{2}(q)}+S_{2}(q)=\frac{\phi(-q^{7/2})\frak f(q^{3/2},q^{11/2})}{q^{3/4}\psi(q^{7})\frak f(-q^2,-q^5)},\notag\\
		\hspace{-2.1cm}(v)&\quad\frac{1}{S_{3}(q)}-S_{3}(q)=\frac{\phi(q^{7/2})\frak f(-q^{5/2},-q^{9/2})}{q^{5/4}\psi(q^{7})\frak f(-q,-q^6)},\notag\\
		\hspace{-2.1cm}(vi)&\quad\frac{1}{S_{3}(q)}+S_{3}(q)=\frac{\phi(-q^{7/2})\frak f(q^{5/2},q^{9/2})}{q^{5/4}\psi(q^{7})\frak f(-q,-q^6)},\notag\\
		\hspace{-2.1cm}(vii)&\quad\left(\frac{1}{S_{1}(q^2)}-S_{1}(q^2)\right)\left(\frac{1}{S_{2}(q^2)}-S_{2}(q^2)\right)\left(\frac{1}{S_{3}(q^2)}-S_{3}(q^2)\right)=\frac{\phi^3(q^{7})\psi(q^{7})}{q^{9/2}\psi^3(q^{14})\psi(q)},\notag\\
		\hspace{-2.1cm}(viii)&\quad \left(\frac{1}{S_{1}(q^2)}+S_{1}(q^2)\right)\left(\frac{1}{S_{2}(q^2)}+S_{2}(q^2)\right)\left(\frac{1}{S_{3}(q^2)}+S_{3}(q^2)\right)=\frac{\phi^3(-q^{7})\psi(-q^{7})}{q^{9/2}\psi^3(q^{14})\psi(q)}.\notag\ \end{align}
\end{thm}
\begin{proof}From \eqref{S1q}, we obtain
	\begin{equation}\label{sq1}\frac{1}{\sqrt{S_{1}(q)}}-\sqrt{S_{1}(q)}=\frac{\frak f(-q^{4},-q^{10})-q^{1/4}\frak f(-q^3,-q^{11})}{\sqrt{q^{1/4}\frak f(-q^3,-q^{11})\frak f(-q^{4},-q^{10})}}.\end{equation}
	From \cite[p. 46, Entry 30 (ii) and (iii)]{bcb3}, we note that
	\begin{equation}\label{fa1}
	\mathfrak{f}(a, b)=\mathfrak{f}(a^3b, ab^3)+a\mathfrak{f}(b/a, a^5b^3).\end{equation}
	Setting $\{a=-q^{1/4}, b=q^{13/4}\}$ and $\{a=q^{1/4}, b=-q^{13/4}\}$ in \eqref{fa1}, we obtain 
	\begin{equation}\label{s2}
	\frak f(-q^{1/4},q^{13/4})=\frak f(-q^{4},-q^{10})-q^{1/4}\frak f(-q^3,-q^{11})
	\end{equation} and
	\begin{equation}\label{s3}
	\frak f(q^{1/4},-q^{13/4})=\frak f(-q^{4},-q^{10})+q^{1/4}\frak f(-q^3,-q^{11}).
	\end{equation}
	Employing \eqref{s2} in \eqref{sq1}, we find that
	\begin{equation}\label{s4}\frac{1}{\sqrt{S_{1}(q)}}-\sqrt{S_{1}(q)}=\frac{\frak f(-q^{1/4},q^{13/4})}{\sqrt{q^{1/4}\frak f(-q^3,-q^{11})\frak f(-q^{4},-q^{10})}}.\end{equation}
	Similarly, from \eqref{S1q} and applying \eqref{s3}, we deduce that
	\begin{equation}\label{s5}\frac{1}{\sqrt{S_{1}(q)}}+\sqrt{S_{1}(q)}=\frac{\frak f(q^{1/4},-q^{13/4})}{\sqrt{q^{1/4}\frak f(-q^3,-q^{11})\frak f(-q^{4},-q^{10})}}.\end{equation}
	Combining \eqref{s4} and \eqref{s5}, we arrive at
	\begin{equation}\label{s6}\frac{1}{{S_{1}(q)}}-{S_{1}(q)}=\frac{\frak f(-q^{1/4},q^{13/4})\frak f(q^{1/4},-q^{13/4})}{q^{1/4}\frak f(-q^3,-q^{11})\frak f(-q^{4},-q^{10})}.\end{equation} Again, from \cite[p. 46, Entry 30 (i),(iv)]{bcb3}, we note that
	\begin{equation}\label{tf2} \mathfrak{f}(a,ab^2)\mathfrak{f}(b,a^2b)=\mathfrak{f}(a,b)\psi(ab)\end{equation} and
	\begin{equation}\label{sr3}
	\mathfrak{f}(a,b)\mathfrak{f}(-a,-b)=\mathfrak{f}(-a^2,-b^2)\phi(-ab).\end{equation} Setting $\{a=-q^3$, $b=-q^{4}\}$  in \eqref{tf2} and $\{a=-q^{1/4}$, $b=q^{13/4}\}$ in \eqref{sr3}, we obtain 
	\begin{equation}\label{s7}
	\frak f(-q^3,-q^{11})\frak f(-q^{4},-q^{10})=\frak f(-q^3,-q^{4})\psi(q^{7})\end{equation} and
	\begin{equation}\label{s8}
	\frak f(-q^{1/4},q^{13/4})\frak f(q^{1/4},-q^{13/4})={\frak f}\left(-q^{1/2},-q^{13/2}\right)\phi(q^{7/2}),\end{equation} respectively.
	Employing \eqref{s7} and \eqref{s8} in \eqref{s6}, we complete the proof of (i).\\
	Squaring \eqref{s5}, we obtain
	\begin{equation}\label{s9}
	\frac{1}{{S_{1}(q)}}+{S_{1}(q)}=\dfrac{\frak f^2(q^{1/4},-q^{13/4})}{q^{1/4}\frak f(-q^3,-q^{11})\frak f(-q^{4},-q^{10})} -2. \end{equation}From \cite[p. 46, Entry 30  (v),(vi)]{bcb3}, we note that
	\begin{equation}\label{as3} \mathfrak{f}^2(a,b)=\mathfrak{f}(a^2,b^2)\phi(ab)+2a\mathfrak{f}(b/a,a^3b)\psi(a^2b^2).\end{equation}Setting $a=q^{1/4}$ and $b=-q^{13/4}$, we obtain
	\begin{equation}\label{s10} \mathfrak{f}^2(q^{1/4},-q^{13/4})=\mathfrak{f}(q^{1/2},q^{13/2})\phi(-q^{7/2})+2q^{1/4}\mathfrak{f}(-q^3,-q^{4})\psi(q^{7}).\end{equation} 
	Employing \eqref{s10}  and  \eqref{s7} in \eqref{s9} and simplifying, we arrive at (ii).
	Proofs of (iii)-(vi) are identical to the proofs of (i) and (ii), so we omit.\\
	From (i), (iii) and (iv), we have
	$$\hspace{-6cm}\left(\frac{1}{S_{1}(q^2)}-S_{1}(q^2)\right)\left(\frac{1}{S_{2}(q^2)}-S_{2}(q^2)\right)\left(\frac{1}{S_{3}(q^2)}-S_{3}(q^2)\right)$$ \begin{equation}\label{y43} =\frac{\phi^3(q^{7})\frak f(-q,-q^{13})\frak f(-q^{3},-q^{11})\frak f(-q^{5},-q^{9})}{q^{9/2}\psi^3(q^{14})\frak f(-q^6,-q^8)\frak f(-q^4,-q^{10})\frak f(-q^2,-q^{12})}
	.\end{equation}
	From \cite[p. 306, Entry 18 (iii),(iv) and (v)]{bcb3}, we have
	\begin{equation}\label{y40}\frak f(q,q^6) \frak f(q^2,q^5) \frak f(q^3,q^4)= \dfrac{ f^2(-q^7)\phi(-q^7)}{\chi(-q)},\end{equation}
	\begin{equation}\label{y41}\frak f(-q,-q^6) \frak f(-q^2,-q^5) \frak f(-q^3,-q^4)= f(-q)f^2(-q^7),\end{equation} and 
	\begin{equation}\label{y42}\frak f(q,q^{13}) \frak f(q^3,q^{11}) \frak f(q^5,q^9)= \chi(q)\psi(-q^7)f^2(-q^{14}).\end{equation} Replacing $q$ by $-q$ in \eqref{y42} and  $q$ by $q^2$ in \eqref{y41} and then applying on \eqref{y43}, we obtain
	\begin{equation}\label{y44}
	\left(\frac{1}{S_{1}(q^2)}-S_{1}(q^2)\right)\left(\frac{1}{S_{2}(q^2)}-S_{2}(q^2)\right)\left(\frac{1}{S_{3}(q^2)}-S_{3}(q^2)\right)=\dfrac{\phi^3(q^7)\psi(q^7)\chi(-q)}{q^{9/2}\psi^3(q^{14})f(-q^2)}.
	\end{equation} Using \eqref{psidef}, \eqref{fqdef} and \eqref{chidef} in \eqref{y44}, we arrive at (vii). Similarly, we can prove (viii). 
\end{proof}
The proof of the Theorem \ref{yy1} is similar to the proof of Theorem \ref{y1}, so we simply state the theorem and omit the proof. 

\begin{thm}\label{yy1}We have 
	\begin{align}
		\hspace{-1.4cm}(i)&\quad\frac{1}{V_{1}(q)}-V_{1}(q)=\frac{\phi(q^{7})\frak f(-q^{6},-q^{8})}{q^{3}\psi(q^{14})\frak f(-q,-q^{13})},\notag\\
		\hspace{-1.4cm}(ii)&\quad\frac{1}{V_{1}(q)}+V_{1}(q)=\frac{\phi(-q^{7})\frak f(q^{6},q^{8})}{q^{3}\psi(q^{14})\frak f(-q,-q^{13})},\notag\\
		\hspace{-1.4cm}(iii)&\quad\frac{1}{V_{2}(q)}-V_{2}(q)=\frac{\phi(q^{7})\frak f(-q^{4},-q^{10})}{q^{2}\psi(q^{14})\frak f(-q^3,-q^{11})},\notag\\
		\hspace{-1.4cm}(iv)&\quad\frac{1}{V_{2}(q)}+V_{2}(q)=\frac{\phi(-q^{7})\frak f(q^{4},q^{10})}{q^{2}\psi(q^{14})\frak f(-q^3,-q^{11})},\notag\\
		\hspace{-1.4cm}(v)&\quad\frac{1}{V_{3}(q)}-V_{3}(q)=\frac{\phi(q^{7})\frak f(-q^{2},-q^{12})}{q\psi(q^{14})\frak f(-q^5,-q^{9})},\notag\\
		\hspace{-1.4cm}(vi)&\quad\frac{1}{V_{3}(q)} +V_{3}(q)=\frac{\phi(-q^{7})\frak f(q^{2},q^{12})}{q\psi(q^{14})\frak f(-q^5,-q^{9})},\notag\\
		\hspace{-1.4cm}(vii) &\quad\left(\frac{1}{V_{1}(q)}-V_{1}(q)\right)\left(\frac{1}{V_{2}(q)}-V_{2}(q)\right)\left(\frac{1}{V_{3}(q)}-V_{3}(q)\right)=\frac{\phi^3(q^{7})\psi(q)}{q^6\psi^3(q^{14})\psi(q^{7})},\notag\\
		\hspace{-1.4cm}(viii)&\quad\left(\frac{1}{V_{1}(q)}+V_{1}(q)\right)\left(\frac{1}{V_{2}(q)}+V_{2}(q)\right)\left(\frac{1}{V_{3}(q)}+V_{3}(q)\right)=\frac{\phi^3(-q^{7})\phi(-q^{14})}{q^6\psi^3(q^{14})\psi(q^{7})\chi(-q^2)\chi(-q)}.\notag\ \end{align}
\end{thm}
\vskip 1mm
\begin{thm}\label{r2}We have
	$$\hspace{-7.4cm}(i)~ V_{1}^n(q)V_{1}^n(-q)=\Big\{\begin{array}{cc}
	V_{1}^n(q^2),\quad if\quad n\equiv0 \pmod{2}\\
	-V_{1}^n(q^2), \quad if\quad n\equiv1 \pmod{2}.
	\end{array}
	$$
	$$\hspace{-11.4cm}(ii)~ V_{2}^n(q)V_{2}^n(-q)=
	V_{2}^n(q^2).$$
	$$\hspace{-7.3cm}(iii)~ V_{3}^n(q)V_{3}^n(-q)=\Big\{\begin{array}{cc}
	V_{3}^n(q^2),\quad if\quad n\equiv0 \pmod{2}\\
	-V_{3}^n(q^2), \quad if\quad n\equiv1 \pmod{2}.
	\end{array}$$	
\end{thm}
\begin{proof}
	From \eqref{V1q}, for any positive integer $n$, 
	\begin{equation}\label{vb} V_{1}^n(q)V_{1}^n(-q)=(-1)^{3n}q^{6n}\dfrac{\mathfrak f^n(-q,-q^{27})}{\mathfrak f^n(-q^{13},-q^{15})}\times\dfrac{\mathfrak f^n(q,q^{27})}{\mathfrak f^n(q^{13},q^{15})}\end{equation} Setting $\{a=q, b=q^{27}\}$ and \{$a=q^{13}, b=q^{15}$\} in \eqref{sr3}, we find that
	\begin{equation}\label{Vabc}\mathfrak{f}(q,q^{27})\mathfrak{f}(-q,-q^{27})=\mathfrak{f}(-q^2,-q^{54})\phi(-q^{28}),\end{equation} and
	\begin{equation}\label{Vabcd}\mathfrak{f}(q^{13},q^{15})\mathfrak{f}(-q^{13},-q^{15})=\mathfrak{f}(-q^{26},-q^{30})\phi(-q^{28}),\end{equation} respectively.
	Employing \eqref{Vabc} and \eqref{Vabcd} in \eqref{vb}, we obtain
	\begin{equation}\label{vcc}V_{1}^n(q)V_{1}^n(-q)
	=(-1)^{3n}q^{6n}\dfrac{\mathfrak f^n(-q^2,-q^{54})\phi^n(-q^{28})}{\mathfrak f^n(-q^{26},-q^{30})\phi^n(-q^{28})}\end{equation} $$=(-1)^{3n}V_{1}^{n}(q^2)$$
	Now the desired result follows from \eqref{vcc} and noting the fact that $3n$ is even if $n\equiv0\pmod{2}$ and odd if $n\equiv1\pmod{2}$. 	
	Proofs of (ii) and (iii) are identical to the proof of (i), so we omit.\end{proof}	

\section{\bf Some partition-theoretic results}
The theta-function identities established in Theorems \ref{y1} and \ref{yy1} can be used to derive some colour partition identities. As example, we derive  colour partition indentities from Theorems \ref{y1}(i) and Theorem \ref{yy1}(i). Similarly,  partition identities can be deduced from remaining theta-function identities given in Theorems \ref{y1} and \ref{yy1}.
First, we give the definition of colour partition of a positive integer $n$ and its generating function. 

A partition of a positive integer $n$ is a non-increasing sequence of positive integers, called parts, whose sum equals $n$. A part in a partition of $n$ is said to have $r$ colours if each part has $r$ copies and all of them are viewed as distinct objects. For any positive integer $n$ and $r$, let $p_{r}(n)$ denote the number of partition of $n$ with each part has $r$ distinct colours. For example, if each part of partition of $3$ has $2$ colours, say white (indicated by the suffix $w$) and black (indicated by the suffix $b$), then the number of $2$ colour partition of $3$ is 10 with the partitions given by 
$3_{w}, \quad 3_{b},\quad 2_{w}+1_{w},\quad 2_{w}+1_{b}, \quad2_{b}+1_{w},\quad 2_{b}+1_{b}, \quad  1_{w}+1_{w}+1_{w},\quad 1_{w}+1_{w}+1_{b},\quad 1_{w}+1_{b}+1_{b},\quad 1_{b}+1_{b}+1_{b}.$ That is,  $p_{2}(3)=10$.\\
The generating function of $p_{r}(n)$ is given by \begin{equation}\label{ww1}\sum_{n=0}^{\infty}p_{r}(n)q^n=\frac{1}{(q;q)^r_{\infty}}.\end{equation} For positive integers $s,m$ and $r$, the quotient
\begin{equation}\label{ww2} \frac{1}{(q^s;q^m)^r_{\infty}}
\end{equation} is the generating function of the number of partitions of $n$ with parts congruent to $s$ modulo $m$ and each part has $r$ colours. For convenience, we use the notation 
\begin{equation}
(q^{r\pm};q^t):=(q^{r},q^{t-r};q^{t})_{\infty},
\end{equation} where $ r $ and $ t $ are positive integers and $r < t$.

\begin{thm}\label{y2}
	Let $C_{1}(n)$ denote the number of partitions of $n$ into parts congruent to $\pm 1, \pm 6, \pm 13$ or $ \pm 14 \pmod {28}$  such that the parts congruent to $\pm 6~and~ \pm 14 \pmod {28}$ have 2 colours. Let $C_{2}(n)$ denote the number of partitions of $n$ into parts congruent to $\pm 1, \pm 8, \pm 13~ or~\pm 14 \pmod {28}$ such that parts congruent to $\pm 8$ and $\pm14 \pmod {28}$ have 2 colours. Let $C_{3}(n)$ denote the number of partitions of $n$ into parts congruent to $\pm 6, \pm 7 ~ and ~ \pm 8 \pmod {28}$ with 2 colours. Then for any positive integer $n\ge 1$,
	$$ C_{1}(n)-C_{2}(n-1)-C_{3}(n)=0.$$ \end{thm}
\begin{proof} Employing \eqref{S1q}, \eqref{phidef},  \eqref{psidef} and replacing $q$ by $q^2$ in Theorem \ref{y1}(i), we obtain
	\begin{equation}\label{y3}
	\frac{(q^{8\pm};q^{28})_{\infty}}{(q^{6\pm};q^{28})_{\infty}}-q\frac{(q^{6\pm};q^{28})_{\infty}}{(q^{8\pm};q^{28})_{\infty}}-\frac{(q^{1\pm, 13\pm};q^{28})_{\infty}(q^{14\pm};q^{28})^2_{\infty}}{(q^{6\pm,8\pm};q^{28})_{\infty}(q^{7\pm};q^{28})^2_{\infty}}=0.\end{equation}
	Dividing \eqref{y3} by $(q^{1\pm,6\pm,8\pm,13\pm};q^{28})_{\infty} (q^{14\pm},q^{48})^2_{\infty}$, we obtain
	\begin{equation}\label{y4} \frac{1}{(q^{6\pm,14\pm};q^{28})^2_{\infty}(q^{1\pm,13\pm};q^{28})_{\infty}}-\frac{q}{(q^{8\pm,14\pm};q^{28})^2_{\infty}(q^{1\pm,13\pm};q^{28})_{\infty}}-\frac{1}{(q^{6\pm,7\pm,8\pm};q^{28})^2_{\infty}}=0.\end{equation}The above quotients of \eqref{y4} represent the generating functions for $C_{1}(n), C_{2}(n)$ and $C_{3}(n)$, respectively. Hence, \eqref{y4} is equivalent to\begin{equation}\label{y5}\sum_{n=0}^{\infty}C_{1}(n)q^n-q\sum_{n=0}^{\infty}C_{2}(n)q^n-\sum_{n=0}^{\infty}C_{3}(n)q^n=0, \end{equation}where we set $C_{1}(0)=C_{2}(0)=C_{3}(0)=1$. Equating coefficients of $q^n$ on both sides, we arrive at the desired result. 
\end{proof}
Theorem \ref{y2} is illustrated in the Table 1 below: 
\begin{table}[h] \caption{The case $n = 7$ of Theorem \ref{y2}} 
	\begin{center} 
		{\renewcommand{\arraystretch}{1}
			\begin{tabular}{|c|c|c|}
				\hline
				$ C_{1}(7)=3 $ & $ C_{2}(6)=1$ & $ C_{3}(7)=2$\\
				\hline
				$6_{r}+1 $ & $1+1+1+1+1+1$ &	$7_{r} $  \\
				\hline
				$6_{g}+1 $ &  &	$7_{g} $  \\
				\hline
				$1+1+1+1+1+1+1$ & &  \\
				\hline	
			\end{tabular}
		}\end{center}
	\end{table}
		
		\begin{thm}\label{y12}
		Let $D_{1}(n)$ denote the number of partitions of $n$ into parts congruent to $\pm 1, \pm 6, \pm 8$ or $ \pm 14 \pmod {28}$  such that the parts congruent to $\pm 1~and~ \pm 14 \pmod {28}$ have 2 colours. Let $D_{2}(n)$ denote the number of partitions of $n$ into parts congruent to $\pm 6, \pm 8, \pm 13~ or~\pm 14 \pmod {28}$ such that parts congruent to $\pm 13$ and $\pm14 \pmod {28}$ have 2 colours. Let $D_{3}(n)$ denote the number of partitions of $n$ into parts congruent to $\pm 1, \pm 7 ~ and ~ \pm 13 \pmod {28}$ with 2 colours. Then for any positive integer $n\ge 6$,
		$$ D_{1}(n)-D_{2}(n-6)-D_{3}(n)=0.$$ \end{thm}
	\begin{proof} Employing \eqref{V1q}, \eqref{phidef},  \eqref{psidef} in Theorem \ref{yy1}(i), we obtain
		\begin{equation}\label{y13} \frac{1}{(q^{1\pm,14\pm};q^{28})^2_{\infty}(q^{6\pm,8\pm};q^{28})_{\infty}}-\frac{q^6}{(q^{13\pm,14\pm};q^{28})^2_{\infty}(q^{6\pm,8\pm};q^{28})_{\infty}}-\frac{1}{(q^{1\pm,7\pm,13\pm};q^{28})^2_{\infty}}=0.\end{equation}The above quotients of \eqref{y13} represent the generating functions for $D_{1}(n), D_{2}(n)$ and $D_{3}(n)$, respectively. Hence, \eqref{y13} is equivalent to\begin{equation}\label{y14}\sum_{n=0}^{\infty}D_{1}(n)q^n-q^6\sum_{n=0}^{\infty}D_{2}(n)q^n-\sum_{n=0}^{\infty}D_{3}(n)q^n=0, \end{equation}where we set $D_{1}(0)=D_{2}(0)=D_{3}(0)=1$. Equating coefficients of $q^n$ on both sides, we arrive at the desired result. 
	\end{proof}
	Theorem \ref{y12} is illustrated in the Table 2 below: 
	\begin{table} [h] \caption{The case $n = 6$ in Theorem \ref{y12}} \begin{center}
			{\renewcommand{\arraystretch}{1}
				\begin{tabular}{|c|c|c|}
					\hline
					$ D_{1}(6)=8 $ & $ D_{2}(0)=1$ & $ D_{3}(6)=7$\\
					\hline
					$6$ &  &	$1_{r}+1_{r}+1_{r}+1_{r}+1_{r}+1_{r} $  \\
					\hline
					$1_{r}+1_{r}+1_{r}+1_{r}+1_{r}+1_{r} $ &  & $1_{r}+1_{r}+1_{r}+1_{r}+1_{r}+1_{g} $\\
					\hline
					$1_{r}+1_{r}+1_{r}+1_{r}+1_{r}+1_{g} $ &  & $1_{r}+1_{r}+1_{r}+1_{r}+1_{g}+1_{g} $ \\
					\hline
					$1_{r}+1_{r}+1_{r}+1_{r}+1_{g}+1_{g} $ &  & $1_{r}+1_{r}+1_{r}+1_{g}+1_{g}+1_{g} $\\
					\hline
					$1_{r}+1_{r}+1_{r}+1_{g}+1_{g}+1_{g} $ &  & $1_{r}+1_{r}+1_{g}+1_{g}+1_{g}+1_{g} $ \\
					\hline
					$1_{r}+1_{r}+1_{g}+1_{g}+1_{g}+1_{g} $ &  & $1_{r}+1_{g}+1_{g}+1_{g}+1_{g}+1_{g} $\\
					\hline
					$1_{r}+1_{g}+1_{g}+1_{g}+1_{g}+1_{g} $ &  & $1_{g}+1_{g}+1_{g}+1_{g}+1_{g}+1_{g} $\\
					\hline
					$1_{g}+1_{g}+1_{g}+1_{g}+1_{g}+1_{g} $ &  & \\
					\hline	
				\end{tabular}
			}\end{center}
		\end{table}	
		
		In next theorem we offer vanishing coefficient in the series expansion of the continued fractions. 
		\begin{thm}\label{c1}
			If$$\dfrac{1}{S^*_{1}(q)}=\frac{(q^4,q^{10};q^{14})_{\infty}}{(q^{3},q^{11};q^{14})_{\infty}}=\sum_{n=0}^{\infty}\alpha_{n}q^n,$$
			$$S^*_{2}(q)=\frac{(q^2,q^{12};q^{14})_{\infty}}{(q^{5},q^{9};q^{14})_{\infty}}=\sum_{n=0}^{\infty}\beta_{n}q^n,$$ 
			$$\dfrac{1}{S^*_{3}(q)}=\frac{(q^{6},q^{8};q^{14})_{\infty}}{(q,q^{13};q^{14})_{\infty}}=\sum_{n=0}^{\infty}\gamma_{n}q^n $$			
			$$V^*_{1}(q)=\frac{(q,q^{27};q^{28})_{\infty}}{(q^{13},q^{15};q^{28})_{\infty}}=\sum_{n=0}^{\infty}\alpha'_{n}q^n,$$
			$$V^*_{2}(q)=\frac{(q^3,q^{25};q^{28})_{\infty}}{(q^{11},q^{17};q^{28})_{\infty}}=\sum_{n=0}^{\infty}\beta'_{n}q^n,$$ 
			$$V^*_{3}(q)=\frac{(q^{5},q^{23};q^{28})_{\infty}}{(q^9,q^{19};q^{28})_{\infty}}=\sum_{n=0}^{\infty}\gamma'_{n}q^n, $$ then 
			$$\hspace{-5cm} (i)\quad \alpha_{7n+1}=0, \qquad\quad  (ii)\quad \beta_{7n+6}=0,  \qquad (iii)\quad \gamma_{7n+6}=0, $$
			$$\hspace{-5cm} (iv)\quad \alpha'_{14n+7}=0, \qquad  (v)\quad \beta'_{14n+4}=0,  \qquad (vi)\quad \gamma'_{14n+11}=0. $$\end{thm}
		\begin{proof} Andrews and Bressoud \cite{geabre} stated the following $p$-dissection formula
			\begin{equation}\label{andw}
			\dfrac{(q^t,q^t,q^{r+s},q^{t-r-s};q^t)_{\infty}}{(q^s,q^{t-s},q^r,q^{t-r};q^t)_{\infty}}=\sum_{j=0}^{p-1}q^{jr}\dfrac{(q^{pt},q^{pt},q^{pr+s+jt},q^{(p-j)t-pr-s};q^{pt})_{\infty}}{(q^{jt+s},q^{(p-j)t-s},q^{pr},q^{(t-r)p};q^{pt})_{\infty}}
			\end{equation} where all of the powers of $q$ in each of the infinite products on the right hand side must be multiples of $p$ and the integer $r$ must satisfy  $ gcd (r,p)=1$.
			
			Now, setting  $t=14$, $r=3$, $s=7$ and $p=7$ in \eqref{andw}, we obtain 
			$$
			\dfrac{(q^{14},q^{14},q^{10},q^{4};q^{14})_{\infty}}{(q^7,q^{7},q^3,q^{11};q^{14})_{\infty}}=\dfrac{(q^{98},q^{98},q^{28},q^{70};q^{98})_{\infty}}{(q^{21},q^{77},q^{7},q^{91};q^{98})_{\infty}}+q^3\dfrac{(q^{98},q^{98},q^{42},q^{56};q^{98})_{\infty}}{(q^{21},q^{77},q^{21},q^{77};q^{98})_{\infty}}$$ $$+q^6\dfrac{(q^{98},q^{98},q^{56},q^{42};q^{98})_{\infty}}{(q^{21},q^{77},q^{35},q^{63};q^{98})_{\infty}}+q^9\dfrac{(q^{98},q^{98},q^{70},q^{28};q^{98})_{\infty}}{(q^{21},q^{77},q^{49},q^{49};q^{98})_{\infty}}+q^{12}\dfrac{(q^{98},q^{98},q^{84},q^{14};q^{98})_{\infty}}{(q^{21},q^{77},q^{63},q^{35};q^{98})_{\infty}}$$ \begin{equation}\label{p1} \hspace{-1.1cm}+q^{15}\dfrac{(q^{98},q^{98},q^{98},q^{0};q^{98})_{\infty}}{(q^{21},q^{77},q^{77},q^{21};q^{98})_{\infty}}+q^{18}\dfrac{(q^{98},q^{98},q^{112},q^{-14};q^{98})_{\infty}}{(q^{21},q^{77},q^{91},q^{7};q^{98})_{\infty}}.
			\end{equation}
			Multiplying both sides of \eqref{p1} by $(q^7;q^{14})^2_{\infty}/(q^{14};q^{14})^2_{\infty}$ and then simplifying, we obtain 
			$$
			\sum_{n=0}^{\infty}\alpha_{n}q^n=\dfrac{(q^7,q^{21},q^{77},q^{91};q^{98})_{\infty}(q^{35},q^{49},q^{63};q^{98})^2_{\infty}}{(q^{28},q^{70};q^{98})_{\infty}(q^{14},q^{42},q^{56},q^{84};q^{98})^2_{\infty}}+q^3\dfrac{(q^7,q^{35},q^{49},q^{63},q^{91};q^{98})^2_{\infty}}{(q^{42},q^{56};q^{98})_{\infty}(q^{14},q^{28},q^{70},q^{84};q^{98})^2_{\infty}}$$ $$+q^6\dfrac{(q^{21},q^{35},q^{63},q^{77};q^{98})_{\infty}(q^{7},q^{49},q^{91};q^{98})^2_{\infty}}{(q^{42},q^{56};q^{98})_{\infty}(q^{14},q^{28},q^{70},q^{84};q^{98})^2_{\infty}}+q^9\dfrac{(q^{21},q^{77};q^{98})_{\infty}(q^{7},q^{35},q^{63},q^{91};q^{98})^2_{\infty}}{(q^{28},q^{70};q^{98})_{\infty}(q^{14},q^{42},q^{56},q^{84};q^{98})^2_{\infty}}$$\begin{equation}\label{diss}+q^{12}\dfrac{(q^{21},q^{35},q^{63},q^{77};q^{98})_{\infty}(q^{7},q^{49},q^{91};q^{98})^2_{\infty}}{(q^{14},q^{84};q^{98})_{\infty}(q^{28},q^{42},q^{56},q^{70};q^{98})^2_{\infty}}+q^4\dfrac{(q^{35},q^{49},q^{63};q^{98})^2_{\infty}}{(q^{14},q^{84};q^{98})_{\infty}(q^{42},q^{56},q^{70},q^{84};q^{98})^2_{\infty}}, 
			\end{equation} where we used the result  $\frak f(-1,a)=0$ from \cite[p. 34, Entry 8(iii)]{bcb3}. Since the right hand side of \eqref{diss} contains no term involving $q^{7n+1}$. So extracting the terms involving $q^{7n+1}$ in \eqref{diss}, we arrive at (i). Proofs of (ii)-(vi) are similar to the proof of (i), so we omit. 
		\end{proof}

	\section*{\bf Acknowledgement} The first author acknowledges the financial support received  from the Department of Science and Technology (DST), Government of India  through INSPIRE Fellowship [DST/INSPIRE Fellowship/2021 /IF210210].

		\end {document}